\documentclass[11pt,reqno]{amsart}
\usepackage{fullpage,a4wide}

\usepackage{graphicx}
\usepackage[draft]{hyperref}
\usepackage{amsmath,amsopn,amssymb,amsfonts,stmaryrd}
\usepackage{verbatim}
\usepackage{amsthm}
\usepackage{mathtools}
\usepackage{color}
\usepackage{enumitem}
\usepackage[framemethod=TikZ]{mdframed}
\usepackage{bbm}
\usepackage{mathrsfs}
\usepackage{booktabs}
\usepackage{caption}
\usepackage{bm}
\usepackage{tensor}
\usepackage{cleveref}
\usepackage{bbm}
\allowdisplaybreaks

\renewcommand{\H}{\mathbb{H}}

\newcommand{\CH}{\mathcal{H}}

\newcommand{\sgn}{\mbox{sgn}}

\newcommand{\SL}{\mathrm{SL}}

\newcommand{\N}{\mathbb N}
\newcommand{\C}{\mathbb C}
\newcommand{\Q}{\mathbb Q}

\DeclareMathAlphabet{\mathdutchcal}{U}{dutchcal}{m}{n}
\SetMathAlphabet{\mathdutchcal}{bold}{U}{dutchcal}{b}{n}
\DeclareMathAlphabet{\mathdutchbcal}{U}{dutchcal}{b}{n}

\theoremstyle{plain}
\newtheorem{thm}{Theorem}[section]
\newtheorem{cor}[thm]{Corollary}
\newtheorem{lem}[thm]{Lemma}

\theoremstyle{definition}

\numberwithin{equation}{section}

\renewcommand{\sgn}{\textnormal{sgn}}

\def\d{\delta}

\def\n{\nu}

\def\d{\delta}

\def\n{\nu}

\def\GG{\Gamma}

\def\GG{\Gamma}

\newcommand{\re}{{\rm Re}}

\renewcommand{\sgn}{{\rm sgn}}
\newcommand{\R}{\mathbb R}
\newcommand{\Z}{\mathbb Z}

\newcommand{\half}{\frac{1}{2}}

\setlist[itemize]{noitemsep, topsep=0pt}

\makeatletter
\newcommand{\vast}{\bBigg@{2}}
\newcommand{\Vast}{\bBigg@{5}}
\makeatother

\usepackage{etoolbox}

\makeatletter
\patchcmd{\@setaddresses}
  {\par\nobreak\begingroup}
  {\par\begingroup}
  {}{\typeout{Patch 1 of \string\@setaddresses\space failed}}

\patchcmd{\@setaddresses}
  {\nobreak\addvspace\bigskipamount}
  {\par\penalty0\addvspace\bigskipamount}
  {}{\typeout{Patch 2 of \string\@setaddresses\space failed}}
\makeatother

\newcommand{\ord}{\operatorname{ord}}
\newcommand{\lcm}{\operatorname{lcm}}

\makeatletter
\renewcommand{\pmod}[1]{\allowbreak\mkern 3mu({\operator@font mod}\mkern 6mu #1)}
\makeatother

\title[Holomorphic eta-quotients with regular sign changes]{Sign changes of Fourier coefficients for holomorphic eta-quotients}
\author{Kathrin Bringmann}
\address{University of Cologne, Department of Mathematics and Computer Science, Weyertal 86-90, 50931 Cologne, Germany}
\email{kbringma@math.uni-koeln.de}
\author{Guoniu Han}
\address{I.R.M.A., UMR 7501, Universit\'e de Strasbourg et CNRS, 7 rue Ren\'e Descartes, F-67084 Strasbourg, France}
\email{guoniu.han@unistra.fr}
\author{Bernhard Heim}
\address{University of Cologne, Department of Mathematics and Computer Science, Weyertal 86-90, 50931 Cologne, Germany}
\email{bheim@uni-koeln.de}
\author{Ben Kane}
\address{The University of Hong Kong, Department of Mathematics, Pokfulam, Hong Kong}
\email{bkane@hku.hk}

\makeatletter
\@namedef{subjclassname@2020}{%
	\textup{2020} Mathematics Subject Classification}
\makeatother
\date{\today}
\begin{document}
\begin{abstract}
	In this paper we study sign changes of an infinite class of $\eta$-quotients that are holomorphic modular forms. We also establish a connection with Hurwitz class numbers.
\end{abstract}
\maketitle

\section{Introduction and statement of results}

There is wide interest in sign changes of the coefficients of $q$-series $f(q) = \sum_{n\in\Z} c(n) q^n$ with $c(n)\in\R$. If $f(q)$ is the Fourier expansion of a cusp form of positive real weight on some congruence subgroup, then Knopp, Kohnen, and Pribitkin \cite[Theorem 1]{KnoppKohnenPribitkin} showed that the $c(n)$ change signs infinitely often. Sign changes along special subsets of $\N$ for half-integral weight cusp forms were considered in \cite{BruinierKohnen,KLW}. Moreover, Kowalski, Lau, Soundararajan, and Wu \cite[Corollary 2]{KLSW10} proved that if $f$ is an integral weight normalized newform, then the signs $\sgn(c(p))$, as $p$ ranges over the primes, uniquely determine $f$. These results do not extend to general holomorphic modular forms. For example, the Eisenstein series for $\SL_2(\Z)$ have at most one sign change, and half of them have no sign change, so the signs of their Fourier coefficients do not uniquely determine them. Define the {\it Dedekind $\eta$-function}
	\begin{equation*}
		\eta(z) := q^{\frac{1}{24}} \prod_{n\geq1}\left(1-q^{n}\right) \qquad \left(q:=e^{2\pi i z}\right).
	\end{equation*}
In this paper, we investigate sign changes of Fourier coefficients of $\eta$-quotients $\prod_{j=1}^{m}\eta(jz)^{\delta_j}$ with $m\in\N$ and $\delta_j\in\Z$ for $1\leq j\leq m$.
Many such quotients are connected to combinatorial counting problems via the
corresponding products
\[
	\prod_{j=1}^{m} \left(q^j;q^j\right)_{\infty}^{\delta_j}=:\sum_{n\geq0}C_{1^{\delta_1} 2^{\delta_2}\cdots m^{\delta_m}}(n)q^n,
\]
where $(a;q)_n:=\prod_{j=0}^{n-1}(1-aq^j)$ for $n\in\N_0\cup\{\infty\}$ is the \begin{it}$q$-Pochhammer symbol\end{it}.
For example, Euler showed that $\frac{1}{\eta(z)}$ is, up to a simple factor, the partition generating function.

Numerous interesting examples involving sign changes of $C_{1^{\delta_1} 2^{\delta_2}\cdots m^{\delta_m}}(n)$ appear in the literature. For example, Andrews and Lewis \cite[Conjecture 2]{AndrewsLewis} made a conjecture pertaining to the so-called crank of a partition that is closely related to $\sgn(C_{1^23^{-1}}(n))$. This conjecture was proved by Kane \cite[Corollary 2]{KaneResolution}. Motivated by work of Borwein \cite{Borwein}, Andrews \cite[Theorem 2.1]{AndrewsBorwein} showed that $\sgn(C_{1^1p^{-1}}(n))$ ($p$ prime) is periodic  with period $p$.  Schlosser and Zhou \cite{Z} further investigated $\sgn(C_{1^{\delta} p^{-\delta}}(n))$ for various choices of $\delta$. In all of these examples, $\sgn(C_{1^{\delta} p^{-\delta}}(n))$ exhibit a regular pattern, in contrast to the case of newforms, where the signs of the Fourier coefficients uniquely determine the form. In this paper, we are interested in other cases for which the signs of $C_{1^{\delta_1}2^{\delta_2}\cdots m^{\delta_m}}(n)$ satisfy a regular pattern. We focus on holomorphic $\eta$-quotients, investigating similar questions for weakly holomorphic modular forms in a parallel paper \cite{BGHK-WeakSigns}. In this paper, we study several types of regular sign patterns. We first present examples for which $\sgn \left( C_{1^{\delta_1}2^{\delta_2}\cdots m^{\delta_m}}(n)\right)$ has some period $M\in\N$.

We start with $M=1$. A number of such examples with fixed sign arise from combinatorial counting problems. For instance, Chen and Garvan \cite[(1.9) and Theorem 2.1]{CG} investigated $\frac{\eta(2z)^{2}\eta(3z)^{3}}{\eta(z)^{2}}$, showing that for $n\in\N_0$ we have
\begin{equation*}
	C_{1^{-2}2^{2}3^{3}}(n)=\frac{1}{24}r_3(24n+11)>0,
\end{equation*}
where $r_3(n)$ is the number of representations of $n$ as a sum of $3$ squares. We obtain a similar result for $\frac{\eta(2z)^{3}\eta(4z)^{2}}{\eta(z)^{2}}$.
\begin{thm}\label{thm:M=1}For $n\in\N_0$, we have
\[
C_{1^{-2}2^{3}4^{2}}(n)>0.
\]
\end{thm}

For $M=2$, we have the example $\frac{\eta(z)^{3}\eta(3z)^{3}}{\eta(2z)^{2}}$.
\begin{thm}\label{thm:132-233}
The sequence $\{\sgn(C_{1^{3}2^{-2}3^{3}}(n))\}_{n\ge1}$ has period $2$. In particular, we have
\begin{equation*}
	(-1)^{n}C_{1^{3}2^{-2}3^{3}}(n)>0.
\end{equation*}
\end{thm}
For $M=3$, we find the example $\frac{\eta(z)^{4}\eta(2z)^{4}}{\eta(3z)^{2}}$.
\begin{thm}\label{thm:14243-2}
The sequence $\{\sgn(C_{1^{4}2^{4}3^{-2}}(n))\}_{n\ge1}$ has period $3$. In particular, we have
\[
\sgn \left(C_{1^{4}2^{4}3^{-2}}(n)\right)=\begin{cases}
1&\text{if }3\mid n,\\
-1&\text{otherwise}.
\end{cases}
\]
\end{thm}
For $M=8$, we find two examples, $\frac{\eta(z)^{4}\eta(2z)^{2}}{\eta(4z)^{2}}$ and $\frac{\eta(z)^{4}\eta(2z)^{4}}{\eta(4z)^{3}}$.
\begin{thm}\label{thm:M=8}
\noindent

\noindent
\begin{enumerate}[leftmargin=*,label=\rm(\arabic*)]
\item
The sequence $\{\sgn(C_{1^{4}2^{2}4^{-2}}(n))\}_{n\ge1}$ has period $8$. In particular, we have
\[
	\sgn\left(C_{1^{4}2^{2}4^{-2}}(n)\right)=
	\begin{cases}
		1&\text{if }n\equiv 0,3,7\pmod{8},\\
		-1&\text{if }n\equiv 1,4,5\pmod{8},\\
		0&\text{if }n\equiv 2\pmod{4}.
	\end{cases}
\]
\item
The sequence $\{\sgn(C_{1^{4}2^{4}4^{-3}}(n))\}_{n\ge1}$ has period $8$. In particular, we have
\[
\sgn \left(C_{1^{4}2^{4}4^{-3}}(n)\right)=\begin{cases}
1&\text{if }n\equiv 0,3,6,7\pmod{8},\\
-1&\text{if }n\equiv 1,2,4,5\pmod{8}.
\end{cases}
\]
\end{enumerate}
\end{thm}
For a prime $p$, we next consider sign changes for the pair of infinite families of $\eta$-quotients
\begin{equation*}
Q_p(z):= \frac{\eta(z)^{p^2}}{\eta(pz)^p}, \quad P_p(z):= \frac{\eta(z)^p}{\eta(pz)}.
\end{equation*}
These families have appeared throughout the literature, and explicit Fourier expansions are known for certain small primes $p$. The function $Q_2$ is an Eisenstein series (see \cite[Example 10.6 (10.14)]{Koehlerbook}) and $P_2$ is a theta function (see \cite[Theorem 1.60]{OnoBook})
\begin{equation*}
Q_2(z) = 1 - 4 \sum_{n\geq1}   (-1)^{n+1} \sum_{ d \vert n } \left( \frac{-1}{d}\right) q^n,\qquad
P_2(z)= \sum_{n\in\Z} (-1)^n q^{n^2},
\end{equation*}
where $\left(\frac{\cdot}{\cdot}\right)$ is the extended Legendre symbol. The Fourier expansion (see \cite[Example 11.4]{Koehlerbook})
\[
P_3(z) = 1 - 3 \sum_{n\geq1} \sum_{d \vert n} \left( \frac{d}{3}\right) q^n + 9 \sum_{n\ge1} \sum_{d \vert n} \left( \frac{d}{3}\right) q^{3n}
\]
closely resembles the Fourier expansion of $Q_2$. These functions are  \begin{it}lacunary\end{it} (see \cite[Theorem 1.2]{CGR06} for the statement for $Q_2$), meaning that the set of $n$ for which the $n$-th Fourier coefficient is non-zero has density zero, and hence their Fourier coefficients cannot exhibit regular sign changes. Although the Fourier expansion (see \cite[Example 12.16 (12.34)]{Koehlerbook})
\begin{equation*}
P_5(z) = 1 - 5 \sum_{n\geq1}  \sum_{d \vert n} \left( \frac{d}{5}\right) d q^n
\end{equation*}
also resembles the Fourier expansions of $Q_2$ and $P_3$, we next see that the signs of the Fourier coefficients of $P_5$ satisfy a somewhat regular pattern related to the prime factorization of $n$ which is part of a more general phenomenon for larger $p$.
\begin{thm}\label{thm:QpPpgen}
\noindent

\noindent
\begin{enumerate}[leftmargin=*,label=\rm(\arabic*)]
\item
Let $p\geq 3$ be prime and $a\in\N_0$. Then, for sufficiently large $m$ coprime to $p$,  we have
\[
\sgn\left(C_{1^{p^2}p^{-p}}\left(p^am\right)\right)=\left(\frac{2}{p}\right) \left(\frac{m}{p}\right).
\]
\item
Let $p\geq 5$ be prime and $a\in\N_0$. Then, for $m$ sufficiently large coprime to $p$,  we have
\[
\sgn\left(C_{1^{p}p^{-1}}\left(p^am\right)\right)=\left(\frac{-2}{p}\right) \left(\frac{m}{p}\right).
\]
\end{enumerate}
\end{thm}

Although the sign changes asserted in Theorem \ref{thm:QpPpgen} only hold for $m$ sufficiently large, in special cases the formula holds for all $n$ if the underlying space of cusp forms is trivial.
\begin{cor}\label{cor:conj99}
\noindent

\noindent
\begin{enumerate}[leftmargin=*,label=\rm(\arabic*)]
\item
For $a\in\N_0$ and $m\in\N$ with $\gcd(3,m)=1$, we have
\[
	\sgn \left(C_{1^93^{-3}}\left(3^am\right)\right)=-\left(\frac{m}{3}\right).
\]
\item
For $a\in\N_0$ and $m\in\N$ with $\gcd(5,m)=1$, we have
\[
	\sgn \left(C_{1^55^{-1}}\left(5^am\right)\right)= -\left(\frac{m}{5}\right).
\]
\end{enumerate}
\end{cor}

Finally, we consider a half-integral weight case for which the sign of the $n$-th Fourier coefficient resembles Corollary \ref{cor:conj99}.

\begin{thm}\label{thm:ClassNum}
For $n\in\N_0$, write $8n+1=3^am$ with $3\nmid m$. Then we have
\[
\sgn \left(C_{1^22^23^{-1}}(n)\right)=
\begin{cases}
\vspace{.15cm}\left(\frac{m}{3}\right)&\text{if }a=0,\\
-\frac{\left(\frac{m}{3}\right)+1}{2} &\text{if }a=1,\\
-1&\text{if }a=2,\\
\sgn \left(C_{1^22^23^{-1}}\left(\frac{n-1}{9}\right)\right)&\text{if }a\geq 3.
\end{cases}
\]
\end{thm}
The paper is organized as follows. In Section \ref{sec:prelim}, we recall some basic facts on modular forms and Hurwitz class numbers. In Section \ref{sec:M=1}, we prove Theorem \ref{thm:M=1}. Section \ref{sec:M=2} is devoted to the proof of Theorem \ref{thm:132-233}. In Section \ref{sec:M=3} we show Theorem \ref{thm:14243-2}, in Section \ref{sec:M=8} we prove Theorem \ref{thm:M=8}, and in Section \ref{sec:conj99} we show \Cref{thm:QpPpgen},  Corollary \ref{cor:conj99}, and \Cref{thm:ClassNum}.

\section*{Acknowledgements}

\noindent We thank the referee for helpful comments. The first author has received funding from the European Research Council (ERC) under the European Union's Horizon 2020 research and innovation programme (grant agreement No. 101001179).
The research of the fourth author was supported by grants from the Research Grants Council of the Hong Kong SAR, China (project numbers HKU 17314122 and HKU 17305923).

\section{Preliminaries}\label{sec:prelim}

\subsection{Modular forms}\label{sec:modularforms}
We briefly recall the relevant definitions and some basic facts about modular forms and refer the reader to \cite{OnoBook} for details. As usual, for $d$ odd, we let
\[
\varepsilon_{d}:=\begin{cases} 1 &\text{if }d\equiv 1\pmod{4}\hspace{-1pt},\\ i&\text{if }d\equiv 3\pmod{4}\hspace{-1pt}.\end{cases}
\]
For $k\in\frac{1}{2}\Z$, $N\in\N$ ($4\mid N$ if $k\in \Z+\frac{1}{2}$), and a character $\chi \pmod N$, a function $f:\H\to\C$ satisfies \begin{it}modularity of weight $k$ on $\Gamma_0(N)$ with character $\chi$\end{it} if for all $\gamma=\left(\begin{smallmatrix}a&b\\ c&d\end{smallmatrix}\right)\in\Gamma_0(N)$
\[
f|_{k}\gamma  = \chi(d) f.
\]
Here the weight $k$ \begin{it}slash operator\end{it} is defined by
\[
f\big|_{k}\gamma(z):=
\begin{cases}
\left( \frac cd \right) \varepsilon_d^{2k}(cz+d)^{-k} f(\gamma z)&\text{if }k\in\Z+\frac{1}{2},\\
(cz+d)^{-k} f(\gamma z)&\text{if }k\in\Z.
\end{cases}
\]
We call $f:\H\to\C$ a \begin{it} holomorphic modular form of weight $k$ on $\Gamma_0(N)$ with character $\chi$\end{it} if $f$ is holomorphic on $\H$, satisfies modularity of weight $k$ on $\Gamma_0(N)$ with character $\chi$, and for every $\gamma=\left(\begin{smallmatrix}a&b\\ c&d\end{smallmatrix}\right)\in\SL_2(\Z)$, $(cz+d)^{-k}f(\gamma z)$ is bounded as $z\to i\infty$. We denote the space of such forms by $M_{k}\left(\Gamma_0(N),\chi\right)$. Modular forms for which $(cz+d)^{-k}f(\gamma z)$ vanishes as $z\to i\infty$ for all $\gamma\in\SL_2(\Z)$ are called \begin{it}cusp forms\end{it}. The corresponding subspace is denoted by $S_{k}(\Gamma_0(N),\chi)$. We drop $\chi$ from the notation if it is trivial. The \begin{it}Petersson inner product\end{it} between $f,g\in S_k(\Gamma_0(N),\chi)$ is defined by ($z=x+iy$)
\[
\left<f,g\right>:=\frac{1}{\left[\SL_2(\Z):\Gamma_0(N)\right]}\int_{\Gamma_0(N)\backslash\H} f(z)\overline{g(z)} y^k\frac{dx dy}{y^2}.
\]

In order to show identities between modular forms, we use the following consequence of the valence formula. 
\begin{lem}\label{lem:valence}
Let $k\in\frac{1}{2}\N$, $N\in\N$, and $\chi$ be a character $\pmod{N}$. If $f\in M_k(\Gamma_0(N),\chi)$ satisfies $c_f(n)=0$ for every $0\leq n\leq N\frac{k}{12}\prod_{p\mid N}(1+\frac{1}{p})$, then $f=0$.
\end{lem}

\subsection{Operators on (non-holomorphic) modular forms}
We recall the action of certain operators on non-holomorphic functions satisfying  modularity of weight $\kappa\in\frac{1}{2}\Z$ on some group $\Gamma_0(N)$ with some character. For $f(z)=\sum_{n\in\Z} c_{f,y}(n) q^n$ (we omit the dependence on $y$ if $f$ is holomorphic), we let
\begin{equation*}
 	f \mid U_\ell(z) := \sum_{n\in\Z} c_{f,\frac{y}{\ell}}(\ell n) q^n, \quad f \mid V_\ell(z) := f(\ell z).
\end{equation*}
For $M\in\N$ and $m\in\N_0$, we define the {\it sieving operator} 
\begin{equation*}
 	f \mid S_{M,m}(z) := \sum_{\substack{n\in\Z\\n\equiv m\pmod M}} c_{f,y}(n)q^n.
\end{equation*}
Moreover,  for a character $\psi$, we define the \begin{it}twist of $f$ by $\psi$\end{it} by
\[
f\otimes \psi(z):=\sum_{n\in\Z} \psi(n)c_{f,y}(n)q^n.
\]
To state the modular properties of these functions, we require the \begin{it}radical\end{it} $\operatorname{rad}(n) :=\prod_{p\mid n}p$ and the character $\chi_D(n):=(\frac {D}{n})$. Recall that the \begin{it}conductor\end{it} of a character $\chi$ (mod $N$) is the smallest $M\in\N$ such that for all $n\in\Z$ with $\gcd(n,N)=1$,  we have $\chi(n+M)=\chi(n)$. The following properties are well-known; see, for example, \cite[Lemma 2.3]{BK} for a proof.
\begin{lem}\label{lem:operatorshalf}
Suppose that $f:\H\to\C$ satisfies modularity of weight $k\in\Z+\frac{1}{2}$ on $\Gamma_0(N)$ $(4 \mid N)$ with character $\chi$ of conductor $N_{\chi}\mid N$.
\noindent

\begin{enumerate}[leftmargin=*, label=\rm(\arabic*)]
\item
For $\delta\in\N$, $f|U_{\delta}$ satisfies modularity of weight $k$ on $\Gamma_0(4\lcm(\frac{N}{4},\operatorname{rad}(\delta)))$ with character $\chi\chi_{4\delta}$.

\item
Suppose that $M\mid 24$ and $M\not\equiv 2\pmod{4}$. Then $f|S_{M,m}$ satisfies modularity of weight $k$ on $\Gamma_{0}(\lcm(N,M^2,MN_{\chi}))$ with character $\chi$.

\item
For $\delta\in\N$, $f|V_{\delta}$ satisfies modularity of weight $k$ on   $\Gamma_0(N\delta)$ with character $\chi\chi_{4\delta}$.
\end{enumerate}
\end{lem}
The following lemma may be shown similarly to \Cref{lem:operatorshalf}.
\begin{lem}\label{lem:operators}
Let $N\in\N$, $\chi$ be a character $\pmod{N}$ with conductor $N_{\chi}\mid N$, and $k\in\N$. Suppose that $f$ satisfies modularity of weight $k$ on $\Gamma_{0}(N)$ with character $\chi$.
	\begin{enumerate}[leftmargin=*,label=\rm(\arabic*)]
		\item For  $\d\in\N$, $f|V_{\d}$ satisfies modularity of weight $k$ on $\Gamma_{0}(N\d)$ with character $\chi$.
		\item If $M\in\N$ with $M\mid 24$ and $m\in\Z$, then $f|S_{M,m}$ satisfies modularity of weight $k$ on
		\noindent$\Gamma_{0}(\lcm(N,M^2,MN_{\chi}))$ with character $\chi$.
		\item For $M\in\N$ and a character $\psi\pmod M$, the twist $f\otimes \psi$ satisfies modularity of weight $k$ for $\Gamma_{0}(\lcm(N,M^2,MN_{\chi}))$ with character $\chi\psi^2$.
		\item If $\delta\mid N$, then $f|U_{\delta}$ satisfies modularity of weight $k$ on $\Gamma_0(N)$ with character $\chi$. 
	\end{enumerate}
\end{lem}

\subsection{Eisenstein series}
For $k\in\N$ with $k\geq 2$ and primitive characters $\chi,\psi$, we define the {\it Eisenstein series}
	\begin{equation}\label{eqn:Ekchipsi}
	E_{k,\chi,\psi}(z):=\delta_{\chi=\chi_1}L\left(1-k,\psi\right)+2\sum_{n\geq1}\sum_{d\mid n}\chi\left(\frac{n}{d}\right)\psi(d) d^{k-1} q^n,
	\end{equation}
where the \begin{it}$L$-function\end{it} for the character $\psi$ is defined for $\re(s)>1$ by $L(s,\psi):=\sum_{n\geq 1} \frac{\psi(n)}{n^s}$. This function admits a meromorphic continuation to $s\in\C$. The modular properties of these Eisenstein series may be found in \cite[Theorem 4.5.2]{DiamondShurman} and \cite[Theorem 4.6.2]{DiamondShurman}.
\begin{lem}\label{lem:Ekchipsi}
Suppose that $\chi$ and $\psi$ are primitive characters of conductors $N_{\chi}$ and $N_{\psi}$, respectively. If $k>2$ or ($k=2$ and either $\chi$ or $\psi$ is non-trivial), then $E_{k,\chi,\psi}\in M_k\left(\Gamma_0\left(N_{\chi}N_{\psi}\right),\chi\psi\right)$.
\end{lem}
For $k,N\in\N$ and a character $\varrho$ (mod $N$), we define the \begin{it}Eisenstein series subspace\end{it} of $M_k(\Gamma_0(N),\varrho)$ to be the space spanned by $E_{k,\chi,\psi}|V_d$, where $\chi$ and $\psi$ are characters satisfying $\chi\psi=\varrho$ and $d\in\N$ satisfies $N_{\chi}N_{\psi}d \mid N$. Here $\chi\psi=\varrho$ means that they agree as characters (mod $N$). Every $f\in M_{k}(\Gamma_0(N),\varrho)$ can be written uniquely as $f=E+g$ with $E$ contained in the Eisenstein series subspace and $g$ a cusp form. We call $E$ the \begin{it}Eisenstein series part\end{it} of $f$ and $g$ the \begin{it}cuspidal part\end{it} of $f$. If $\chi=\chi_1$ and $\psi=\chi_{(\frac{-1}{p})p}$ for an odd prime $p$, then we normalize the Eisenstein series by multiplying by
\[
L_{k,p}:=\frac{1}{L\left(1-k,\chi_{\left(\frac{-1}{p}\right)p}\right)}.
\]
Using the Euler--Maclaurin summation formula (see \cite[(44)]{ZagierMellin}), one can determine the behavior of the Eisenstein series at $0$.
\begin{lem}\label{lem:Eisensteinp-growth}
Suppose that $k\geq 2$ and $p$ is an odd prime. Then $L_{k,p} E_{k,\chi_1,\chi_{(\frac{-1}{p})p}}$ has constant term $1$ at $i\infty$ and vanishes at $0$.
\end{lem}
In addition to the holomorphic Eisenstein series, we also require the quasimodular Eisenstein series of weight $2$, which, for $\sigma(n):=\sum_{d\mid n} d$, is defined by
\begin{equation*}
E_2(z):=1-24\sum_{n\geq1}\sigma(n)q^n.
\end{equation*}
The Eisenstein series $E_2$ is not modular, but it has a natural ``modular completion''
\[
\widehat{E}_2(z):=E_2(z)-\frac{3}{\pi y}.
\]
Although $\widehat{E}_2$ is not holomorphic, it satisfies, for all $\gamma\in\SL_2(\Z)$,
\[
\widehat{E}_2|_{2}\gamma = \widehat{E}_2.
\]

\subsection{Hecke operators}
In this subsection, we restrict to integral weight modular forms. For $N,k\in\N$, $\chi$ a character (mod $N$), and $p$ a prime, we define the \begin{it}Hecke operator\end{it} $T_p$ acting on $f\in M_{k}(\Gamma_0(N),\chi)$ by  (see \cite[Definition 2.1]{OnoBook})
\[
f\big|T_p(z):=\sum_{n\geq0}\left(c_f(pn)+\chi(p)p^{k-1}c_f\left(\frac{n}{p}\right)\right)q^n,
\]
where $c_f(\alpha):=0$ for $\alpha\in\Q^{+}\setminus\N_0$. There exists a natural basis of cusp forms which are simultaneous eigenfunctions under all Hecke operators $T_p$ with $p\nmid N$. We call these simultaneous eigenfunctions \begin{it}Hecke eigenforms\end{it}. If $f\in M_{k}(\Gamma_0(N),\chi)$ is a Hecke eigenform, then $f|V_d\in M_{k}(\Gamma_0(Nd),\chi)$ is also a Hecke eigenform. For $M\in\N$, the subspace of $M_{k}(\Gamma_0(M),\chi)$ spanned by $f|V_d$ with $f\in M_{k}(\Gamma_0(N),\chi)$ for $1\leq N<M$ is the \begin{it}old space\end{it}; the orthogonal complement of the old space is the \begin{it}new space\end{it}. Hecke eigenforms in the new spaces are called \begin{it}newforms\end{it}. We normalize the Fourier expansions so that $c_f(1)=1$. Letting $d(n)$ denote the number of divisors of $n$, a result of Deligne \cite{Deligne} gives an explicit bound on $c_f(n)$.

\begin{thm}\label{lem:Deligne}(Deligne)
Suppose that $k$, $N\in\N$, $\chi$ is a character $\pmod N$, and $f\in S_{k}(\Gamma_0(N),\chi)$ is a normalized newform. Then
	\begin{equation*}
		\left|c_f(n)\right| \le d(n) n^{\frac{k-1}2}.
	\end{equation*}
\end{thm}
Letting $\|f\|:=\sqrt{\left<f,f\right>}$ denote the Petersson norm, Schulze-Pillot and Yenirce \cite[Theorem 12]{SY} constructed an explicit orthonormal basis for $S_{k}(\Gamma_0(N),\chi)$ and used \Cref{lem:Deligne} to obtain a bound for the Fourier coefficients of any $f\in S_k(\Gamma_0(N),\chi)$.
\begin{lem}\label{lem:SY}
Suppose that $k,N\in\N$, $\chi$ is a character $\pmod N$, and $f\in S_{k}(\Gamma_0(N),\chi)$. Then we have
\[
\left|c_f(n)\right|\leq 2\sqrt{\pi} e^{2\pi}\sqrt{\dim_{\C} (S_k\left(\Gamma_0(N),\chi\right))} \sqrt N\prod_{p\mid N}\frac{\left(1+\frac{1}{p}\right)^3}{\sqrt{1-\frac{1}{p^4}}}  \|f\|d(n)n^{\frac{k-1}{2}}.
\]
\end{lem}

\subsection{$L$-functions of Dirichlet characters}
We require a lemma for evaluating $L$-functions of Dirichlet characters at non-positive integers. The following identity is well-known and follows directly from \cite[p. 249]{Apostol} and \cite[Theorem 12.13]{Apostol}.
\begin{lem}\label{lem:Dirichlet}
	For a discriminant $D$ and $k\in\N$, we have
	\[
		L\left(1-k,\chi_D\right)= -\frac{|D|^{k-1}}{k}\sum_{r=1}^{|D|} \chi_D(r)B_k\left(\frac{r}{|D|}\right),
	\]
	where $B_k(x)$ denotes the $k$-th Bernoulli polynomial.
\end{lem}

For the specific case $D=(\frac{-1}{p}) p$ and $k\equiv k_p\pmod{2}$ with $k_p:=\frac{p-1}{2}$, one may use the functional equation for the Dirichlet $L$-function and Gauss's evaluation of the quadratic Gauss sum to obtain the sign of the $L$-value in \Cref{lem:Dirichlet}.
\begin{lem}\label{lem:Lchisgn}
Let $p$ be an odd prime and $k\geq 2$ be an integer satisfying $k\equiv k_p\pmod{2}$. Then
\[
\sgn\left(L_{k,p}\right)=(-1)^{\frac{k}{2}+\frac{p-1}{4}}\left(\frac{-2}{p}\right).
\]
In particular, we have
\[
\sgn\left(L_{k_p,p}\right)=\left(\frac{-2}{p}\right)\quad\text{ and }\quad \sgn\left(L_{pk_p,p}\right)=\left(\frac{2}{p}\right).
\]

\end{lem}
\subsection{Hurwitz class numbers}\label{sec:Hurwitz}

For $D\in\N$, let $H(D)$ denote the {\it Hurwitz class number}, the weighted number of classes of positive-definite integral binary quadratic forms of discriminant $-D$, where each class is weighted by the inverse of the size of its automorphism group in $\operatorname{PSL}_{2}(\Z)$.
We note that if $-D<0$ is a discriminant, then $H(D)>0$, while $H(D)=0$ if $-D$ is not a discriminant. We extend the definition by setting $H(0):=-\frac{1}{12}$ and $H(r):=0$ for $r\in \Q\setminus\N_0$.
For a fundamental discriminant $-D$, we have \cite[p. 273]{Cohen}
	\begin{equation}\label{E:Hform}
	 	H\left(Df^2\right) = H(D) S_{D}(f).
	\end{equation}
	Here, letting $\mu$ denote the M\"obius function, we define
	\begin{equation}\label{eqn:SDfdef}
		S_D(f) := \sum_{d\mid f} \mu(d) \chi_{-D}(d) \sigma\left(\frac {f}d\right).
	\end{equation}
For $\ell_1,\ell_2\in\N$, let
\begin{equation}\label{eqn:Hslashdef}
 	\mathcal H(z) := \sum_{D\geq 0} H(D) q^D, \quad \mathcal H_{\ell_1,\ell_2} := \mathcal H \mid \left(U_{\ell_1\ell_2} - \ell_2 U_{\ell_1} \circ V_{\ell_2}\right).
\end{equation}
For $y>0$, let $\GG(s,y):=\int_{y}^{\infty}e^{-t} t^{s-1}dt$ be the {\it incomplete gamma function}. The modularity of
	\begin{equation*}
		\widehat{\CH}(z):=\CH(z)+\frac{1}{8\pi \sqrt{y}}+ \frac{1}{4\sqrt{\pi}}\sum_{n\ge1}n\GG\left(-\half,4\pi n^{2} y\right)q^{-n^{2}}
	\end{equation*}
was proved by Zagier \cite{Zagier} (see also \cite[Chapter 2, Theorem 2]{HZ}). Using his result, the modularity of $\mathcal{H}_{\ell_1,\ell_2}$ was given in \cite[Lemma 2.6]{BK}.
\begin{lem}\label{lem:Hslash}
	For $\ell_1,\ell_2\in\N$, with $\gcd(\ell_1,\ell_2)=1$ and $\ell_2$ squarefree, we have that
	\begin{equation*}
	\mathcal{H}_{\ell_1,\ell_2}\in  M_{\frac{3}{2}}\left(\Gamma_0(4\operatorname{rad}(\ell_1)\ell_2),\chi_{4\ell_1 \ell_2}\right).
	\end{equation*}
\end{lem}

\section{Proof of Theorem \ref{thm:M=1}}\label{sec:M=1}

We abbreviate $b_{1}(n) := C_{1^{-2}2^{3}4^{2}}(n)$.
\begin{proof}[Proof of Theorem \ref{thm:M=1}]
By \cite[Theorem 1.60]{OnoBook}, we have
\[
	\frac{\eta(16z)^2}{\eta(8z)} = \sum_{n\geq0} q^{(2n+1)^2}.
\]
We therefore have\footnote{Throughout we use boldface letters for vectors.} 
\begin{equation*}
	\sum_{n\geq0}b_{1}(n) q^{8n+4} = \sum_{\bm{n}\in\N_0^3} q^{8\left(T_{n_1} + T_{n_2} + 2T_{n_3}\right) + 4},
\end{equation*}
where $T_n:=\frac{n(n+1)}{2}$. Comparing the $(8n+4)$-th Fourier coefficient on both sides, we see that
\begin{equation}\label{eqn:c(1,-2),(2,3),(4,2)}
	b_{1}(n)=\#\left\{\bm{n}\in\N_0^3: T_{n_1}+T_{n_2}+2T_{n_3}=n\right\}.
\end{equation}

It was proved by Liouville \cite{Liouville} (see also the statement in \cite{BosmaKane} and its generalization in \cite[Theorem 1.1]{BosmaKane}) that, for $n\in\N_0$,
\[
\#\left\{\bm{n}\in\N_0^3: T_{n_1}+T_{n_2}+2T_{n_3}=n\right\}>0.
\]
Combining this with \eqref{eqn:c(1,-2),(2,3),(4,2)} immediately yields Theorem  \ref{thm:M=1}.
\end{proof}

\section{Proof of Theorem \ref{thm:132-233}}\label{sec:M=2}

Abbreviating $b_{2}(n) := C_{1^{3}2^{-2}3^{3}}(n)$, $s_{2}(n):=\sgn(b_{2}(n))$, we first show the following.
\begin{lem}\label{lem:c132-233}
	We have
\[
	b_{2}(n)=
	\begin{cases}
		\sum_{d\mid (3n+1)} \left(\frac3d\right)d &\text{if }n\equiv 0\pmod{4},\\
		-\frac{1}{3}\sum_{d\mid (3n+1)} \left(\frac3d\right)d &\text{if }n\equiv 2\pmod{4},\\
		-\frac{1}{3}\sum_{d\mid (3n+1)} \left(\frac{12}d\right)d-\frac{2}{3}\sum_{d\mid (3n+1)} \left(\frac{12}{\frac{3n+1}{d}}\right)d &\text{if }n\equiv 1\pmod{2}.
	\end{cases}
\]
\end{lem}
\begin{proof}
	By \cite[Theorem 1.64]{OnoBook}, $\frac{\eta(3z)^3\eta(9z)^3}{\eta(6z)^2}\in M_2(\Gamma_0(72),\chi_{12})$.

	The generating function of the right-hand side of Lemma \ref{lem:c132-233} is
	\[
		\frac{1}{2}E_{2,\chi_1,\chi_{12}}\big|S_{12,1} -\frac{1}{6}E_{2,\chi_1,\chi_{12}}\big|S_{12,7} -\frac{1}{6}E_{2,\chi_1,\chi_{12}}\big|S_{6,4} -\frac{1}{3}E_{2,\chi_{12},\chi_{1}}\big|S_{6,4}.
	\]
By Lemma \ref{lem:Ekchipsi}, $E_{2,\chi_{1},\chi_{12}},E_{2,\chi_{12},\chi_{1}}\in M_{2}\left(\Gamma_0(12),\chi_{12}\right)$.
Hence, for $m\in\Z$ and $M\in\N$ with $M\mid 12$, Lemma \ref{lem:operators} (2) implies that $E_{2,\chi_1,\chi_{12}}\big|S_{M,m}, E_{2,\chi_{12},\chi_{1}}\big|S_{M,m} \in M_{2}\left(\Gamma_0(144),\chi_{12}\right)$. We conclude that the generating functions of both sides of the claimed identity lie in $M_2(\Gamma_0(144),\chi_{12})$. By \Cref{lem:valence}, the identity holds if it holds for the first $48$ Fourier coefficients. It was verified computationally for all $3n+1\leq 301$.
\end{proof}
\begin{proof}[Proof of Theorem \ref{thm:132-233}]
	Writing $3n+1=2^{\ell}m$ with $\gcd(6,m)=1$, we have
	\begin{equation*}
		\sum_{d\mid (3n+1)}\hspace{-0.15cm} \left(\frac{12}{d}\right)d = \sum_{d\mid m} \left(\frac{3}{d}\right)d \sum_{r=0}^{\ell} \left(\frac{12}{2^{r}}\right)2^{r} = \prod_{p\mid m}\sum_{j=0}^{\ord_p(m)} \left(\frac{3}{p^{j}}\right)p^{j} = \prod_{p\mid m} \frac{1-\left(\left(\frac{3}{p}\right)p\right)^{\ord_p(m)+1}}{1-\left(\frac{3}{p}\right)p}.
	\end{equation*}
	Now
	\begin{equation*}
		\frac{1-\left(\left(\frac{3}{p}\right)p\right)^{\ord_p(m)+1}}{1-\left(\frac{3}{p}\right)p} = \begin{cases}
			\frac{1-p^{\ord_p(m)}}{1-p} & \text{if}\ p\equiv \pm 1 \pmod{12},\vspace{.1cm}\\
			\frac{1-p^{\ord_p(m)}}{1+p} & \text{if}\ p\equiv \pm 5 \pmod{12},\ \ord_p(m) \ \text{odd},\vspace{.1cm}\\
			\frac{1+p^{\ord_p(m)}}{1+p} & \text{if}\ p\equiv \pm 5 \pmod{12},\ \ord_p(m) \ \text{even}.\\
		\end{cases}
	\end{equation*}
	This implies that
	\begin{equation}\label{eqn:sgnm}
		\sgn\left(\sum_{d\mid m}\left(\frac{3}{d}\right)d\right)=\left(\frac{3}{m}\right) = (-1)^{\frac{m-1}{2}}\left(\frac{m}{3}\right).
	\end{equation}

	For $n$ even, we have $3n+1=m$, so \eqref{eqn:sgnm} implies that
	\begin{equation}\label{eqn:sgneven}
	\sgn\left(\sum_{d\mid (3n+1)} \left(\frac{3}{d}\right)d\right) = (-1)^{\frac{m-1}{2}}\left(\frac{m}{3}\right).
	\end{equation}

	If $n\equiv 0\pmod{4}$, then $m=3n+1\equiv 1\pmod{12}$, so plugging \eqref{eqn:sgneven} into Lemma \ref{lem:c132-233} yields that $s_{2}(n)=1$.
	If $n\equiv 2\pmod{4}$, then $m=3n+1\equiv 7\pmod{12}$, so inserting \eqref{eqn:sgneven} into Lemma \ref{lem:c132-233} implies that $s_{2}(n)=1$. For $2\nmid n$, we write $3n+1=2^{\ell}m$ with $m$ odd. Noting that $(\frac{12}{\frac{3n+1}{d}})=0$ unless $2^{\ell}\mid d$, Lemma \ref{lem:c132-233} gives that
	\begin{equation*}
		b_{2}(n) = -\frac{1}{3}\left(1+2^{\ell+1}\left(\frac{3}{m}\right)\right)\sum_{d\mid m}\left(\frac{3}{d}\right) d.
	\end{equation*}
	By \eqref{eqn:sgnm}, we have $\sgn(\sum_{d\mid m}\left(\frac{3}{d}\right) d)=(\frac3m)$. Since $2^{\ell+1}>1$, $\sgn(1+2^{\ell+1}(\frac{3}{m}))=\sgn((\frac{3}{m}))$. So overall we obtain $s_{2}(n) = -1$ for $n$ odd. Combining these cases proves the claim. 
\end{proof}

\section{Proof of Theorem \ref{thm:14243-2}}\label{sec:M=3}

\subsection{Decomposition of the eta-quotient}
We now explicitly decompose $\frac{\eta(4z)^4\eta(8z)^4}{\eta(12z)^2}$ into an Eisenstein series and a cusp form. For this, we define the Eisenstein series
\begin{align*}
	\mathcal{E}(z)&:=\sum_{\substack{n\geq 1\\ n\equiv 1\pmod{12}}} \sum_{d\mid n}\left(\frac{-1}{d}\right)d^2  q^n-\frac{1}{2}\sum_{\substack{n\geq 1\\ n\equiv 5\pmod{12}}} \sum_{d\mid n}\left(\frac{-1}{d}\right)d^2  q^n\\
	&\qquad - 2\sum_{\substack{n\geq 1\\ n\equiv 9\pmod{12}}} \left(\sum_{d\mid n}\left(\frac{-1}{d}\right)d^2 +6\sum_{d\mid \frac{n}{3}}\left(\frac{-1}{d}\right)d^2 \right) q^n\\
&=\frac{1}{2}E_{3,\chi_1,\chi_{-4}}\big|S_{12,1}-\frac{1}{4} E_{3,\chi_1,\chi_{-4}}\big|S_{12,5}- E_{3,\chi_1,\chi_{-4}}\big|S_{12,9}-6E_{3,\chi_1,\chi_{-4}}\big|V_{3} \big|S_{4,1}.
\end{align*}
For the cuspidal part, let $g_{1}$ denote the normalized newform in $S_{3}(\Gamma_0(12),\chi_{-4})$ with
\begin{align*}
g_{1}(z)=q-\left(1+\sqrt{3}i\right) q^2 + \sqrt{3}i q^3 -2\left(1-\sqrt{3}i\right) q^4 -2q^5+\left(3-\sqrt{3}i\right) q^6 &-4\sqrt{3}i q^7 +8q^8\\
&\hspace{.4cm}-3q^9+O\left(q^{10}\right)
\end{align*}
 and let $g_{2}$ be the normalized newform in $S_{3}(\Gamma_0(36),\chi_{-4})$ satisfying
\[
	g_{2}(z)=q-2q^2+4q^4+8q^5-8q^8-16q^{10}-10q^{13}+16q^{16}-16q^{17}+O\left(q^{20}\right).
\]
We then obtain the following decomposition of the eta-quotient.
\begin{lem}\label{lem:14243-2}
We have
\begin{multline*}
	\frac{\eta(4z)^4\eta(8z)^4}{\eta(12z)^2}= \frac{1}{7}\mathcal{E}(z) -\frac{27}{14}g_{1}\big|S_{12,9}(z) +\frac{3}{14}g_{1}\big|S_{4,1}(z)+\frac{9}{14}g_{1}\otimes \chi_{-3}\big|S_{4,1}(z)\\
-\frac{3}{16}g_{2}\big|S_{4,1}(z)+\frac{3}{16}g_{2}\otimes \chi_{-3}\big|S_{4,1}(z).
\end{multline*}
\end{lem}
\begin{proof}
	By \cite[Theorem 1.64]{OnoBook}, ${\frac{\eta(4z)^4\eta(8z)^4}{\eta(12z)^2}\in M_3(\Gamma_0(72),\chi_{-4})}$.
	By \Cref{lem:Ekchipsi}, $E_{3,\chi_1,\chi_{-4}}\in M_3(\Gamma_0(4),\chi_{-4})$. Using Lemma \ref{lem:operators} (2), $E_{3,\chi_1,\chi_{-4}}|S_{12,m}\in M_3(\Gamma_0(144),\chi_{-4})$ for $m\in\Z$. Lemma \ref{lem:operators} (1) implies that $E_{3,\chi_1,\chi_{-4}}|V_3\in M_3(\Gamma_0(12),\chi_{-4})$. Lemma \ref{lem:operators} (2) then gives that $E_{3,\chi_1,\chi_{-4}}|V_3|S_{4,1}\in M_3(\Gamma_0(48),\chi_{-4})$. Thus $\mathcal{E}\in M_{3}\left(\Gamma_0(144),\chi_{-4}\right)$. By Lemma \ref{lem:operators} (2), \!(3), 
	\begin{equation*}
		-\tfrac{27}{14}g_{1}\big|S_{12,9}+\tfrac{3}{14}g_{1}\big|S_{4,1}+\tfrac{9}{14}g_{1}\otimes\chi_{-3}\big|S_{4,1}-\tfrac{3}{16}g_{2}\big|S_{4,1}+\tfrac{3}{16}g_{2}\otimes \chi_{-3}\big|S_{4,1} \in M_{3}\left(\Gamma_0(144),\chi_{-4}\right).
	\end{equation*}
	Thus both sides of the claimed identity lie in $M_3(\Gamma_0(144),\chi_{-4})$. To prove the identity, it suffices to check the first $72$ Fourier coefficients. This was verified numerically.
\end{proof}

\subsection{Fourier coefficients of the Eisenstein series part}
We now determine the sign of the $n$-th Fourier coefficient $A(n)$ of the Eisenstein series part $\mathcal E$.
\begin{lem}\label{lem:E72sgn}
	Suppose that $n\equiv 1\pmod{4}$. Then we have
	\[
		\sum_{d\mid n}\left(\frac{-1}{d}\right)d^2+6\sum_{d\mid \frac{n}{3}}\left(\frac{-1}{d}\right)d^2>0.
	\]
In particular,
\[
\sgn(A(4n+1)) = \begin{cases} 1&\text{if }3\mid n,\\
-1&\text{otherwise.}
\end{cases}
\]
\end{lem}
	\begin{proof}
	If $3\nmid n$, then the second sum on the left-hand side of the first claim vanishes. Since $n\equiv1\pmod{4}$, we have that $n$ (and thus $d$) is odd. Thus the above becomes
\begin{equation*}
	\prod_{p\mid n} \sum_{j=0}^{\ord_p(n)} \left(\frac{-1}{p^j}\right) p^{2j} = \prod_{p\mid n} \begin{cases}
		\frac{1-p^{2\ord_p(n)+2}}{1-p^2} & \text{if $p\equiv1\pmod4$},\vspace{0.1cm}\\
		\frac{1-p^{2\ord_p(n)+2}}{1+p^2} & \text{if}\ p\equiv3\pmod4,\, 2\nmid \ord_p(n),\vspace{.1cm}\\
		\frac{1+p^{2\ord_p(n)+2}}{1+p^2} & \text{if}\ p\equiv3\pmod4,\, 2\mid\ord_p(n).
	\end{cases}
\end{equation*}
Now only the second case yields a negative sign, so
\begin{equation}\label{eqn:nmod4}
\sgn\left(\sum_{d\mid n}\left(\frac{-1}{d}\right)d^2\right)=
	(-1)^{\#\{p\mid n\,:\,p\equiv3\pmod4,\,2\nmid\ord_p(n)\}}\equiv n \pmod{4}
\end{equation}
since $n$ is odd.
Since $n\equiv 1\pmod{4}$, we conclude that the sign in this case is $1$.

Finally, suppose that $3\mid n$.
Write $n=3^\ell m$ ($\ell\in\N$), $3\nmid m$. Then we need to rewrite the left-hand side of the first claim of the lemma as
\begin{multline}\label{eqn:extrafactor}
	\sum_{d\mid m} \left(\frac{-1}d\right) d^2 \sum_{j=0}^\ell \left(\frac{-1}{3^j}\right) 3^{2j} + 6\sum_{d\mid m} \left(\frac{-1}d\right) d^2 \sum_{j=0}^{\ell-1} \left(\frac{-1}{3^j}\right) 3^{2j}\\ =  \left(7\sum_{j=0}^{\ell-1} (-1)^j 3^{2j} + (-1)^\ell 3^{2\ell}\right)\sum_{d\mid m} \left(\frac{-1}d\right) d^2.
\end{multline}
The sign of the sum over $d$ is evaluated in \eqref{eqn:nmod4}. The expression in parentheses equals $\frac{7 + (-1)^{\ell}3^{2\ell +1}}{10}$. 
Note that the sign of this 
 factor is $(-1)^{\ell}$ (because $\ell\in\N$). Combining this with the above, we have, for $3\mid n$,
\begin{equation*}
	\sgn\left(\sum_{d\mid n}\left(\frac{-1}{d}\right)d^2+6\sum_{d\mid \frac{n}{3}}\left(\frac{-1}{d}\right)d^2\right) = (-1)^{\#\{p\mid n\,:\,p\equiv3\pmod4,\,2\nmid\ord_p(n)\}}.
\end{equation*}
Since $n\equiv 1\pmod{4}$, \eqref{eqn:nmod4} again implies that the sign is $1$. Substituting the first claim into the definition of $\mathcal{E}$ yields the second claim.
\end{proof}
\subsection{Finishing the proof of Theorem \ref{thm:14243-2}}
By Lemma \ref{lem:E72sgn}, the signs of the Fourier coefficients of $\mathcal{E}$ agree with those asserted in \Cref{thm:14243-2}. We are now ready to prove Theorem \ref{thm:14243-2}. Let $b_{3}(n) := C_{1^{4}2^{2}3^{-2}}(n)$ and $s_{3}(n) :=\sgn(b_{3}(n))$. 
\begin{proof}[Proof of Theorem \ref{thm:14243-2}]
We claim that,  for $n\in\N$, $s_{3}(n) = \sgn(A(4n+1))$. \Cref{lem:E72sgn} then gives the claim. To show this, we explicitly compare the growth of the Fourier coefficients of $\mathcal{E}$ with the growth of the Fourier coefficients of the cuspidal part from \Cref{lem:14243-2}. We begin by bounding the Fourier coefficients of the cuspidal part of $\frac{\eta(4z)^4\eta(8z)^4}{\eta(12z)^2}$.
	Write
	\begin{equation*}
		g_{1}(z) =: \sum_{n\ge1} a_{1}(n) q^n, \quad g_{2}(z) =: \sum_{n\ge1} a_{2}(n) q^n.
	\end{equation*}
	Since $g_{1}$ and $g_{2}$ are weight $3$ newforms, $|a_{1}(n)|$, $|a_{2}(n)|\le d(n)n$ by \Cref{lem:Deligne}.  Moreover, by \Cref{lem:14243-2}, for $n\equiv 1\pmod{4}$ the $n$-th Fourier coefficient of the cuspidal part of the function on the right-hand side of \Cref{lem:14243-2} is
	\begin{equation}\label{cusp}
		-\frac{12}{7}\delta_{3\mid n}a_{1}(n) + \delta_{3\nmid n}\frac{3}{14} \left(1+3\left(\frac{-3}n\right)\right) a_{1}(n) + \frac{3}{16}\left(-1+\left(\frac{-3}n\right)\right) a_{2}(n).
	\end{equation}
	Note that $(\frac{-3}n)=(\frac n3)$. We now consider the residue classes of $n$ (mod $12$).

	Assume first that $n\equiv1\pmod{12}$. Then \eqref{cusp} equals $\frac67a_{1}(n)$ and its absolute value can be bounded by $\frac{6}7d(n)n$. From \Cref{lem:14243-2}, the absolute value of the $n$-th Fourier coefficient of the Eisenstein series part $\mathcal{E}$ is
	\begin{equation*}
		\frac17\left|\sum_{d\mid n}\left(\frac{-1}d\right)d^2\right| \geq \frac17 \prod_{p\mid n} \frac{p^{2\ord_p(n)+2}-1}{p^2+1}.
	\end{equation*}
We conclude that $s_{3}(n)$ agrees with the claimed value if
	\[
		d(n) \le \frac16 \prod_{p\mid n} \frac{p^{2\ord_p(n)+2}-1}{p^{\ord_p(n)}\left(p^2+1\right)}.
	\]
	Writing $n=\prod_{j=1}^{r}p_{j}^{\ell_{j}}$, we have $d(n)=\prod_{j=1}^{r}(\ell_{j}+1)$. Hence the above is equivalent to
	\begin{equation}\label{eqn:ratbound}
		\prod_{j=1}^{r}(\ell_{j}+1)\le \frac16 \prod_{j=1}^{r} \frac{p_{j}^{2\ell_{j}+2}-1}{p_{j}^{\ell_{j}}\left(p_{j}^{2}+1\right)}. 
	\end{equation}
	We claim that, for a prime $p$ and $\ell\in\N$,
	\begin{equation}\label{E:claimcases}
		\frac{p^{2\ell+2}-1}{p^{\ell} \left(p^2+1\right)} \ge
		\begin{cases}
			2.4(\ell + 1) & \text{for $p=5$ and $\ell=1$},\\
			3.4(\ell + 1) & \text{for $p=7$ and $\ell=1$},\\
                        5(\ell+1)&\text{for $p=11$ and $\ell=1$},\\
			6(\ell+1) & \text{for $p\ge13$ or ($p\in\{5,7,11\}$ and $\ell\ge2$)}.
		\end{cases}
	\end{equation}
	This clearly implies the claim unless $n$ is one of the exceptional cases $n\in\{5,7,11\}$. But none of these satisfy $n\equiv1\pmod{12}$. We now prove \eqref{E:claimcases}. We first directly verify \eqref{E:claimcases} for $\ell=1$ and $p\in\{5,7,11\}$ by computing both sides. The claim for the remaining cases in \eqref{E:claimcases} follows after  showing that for $x\ge13$ or ($x\in\{5,7,11\}$ and $\ell\geq 2$)
	\begin{equation*}
		f_\ell(x) := x^{2\ell+2} - 1 - 6(\ell+1) x^{\ell}\left(x^2+1\right) > 0.
	\end{equation*}
	We do so by induction on $\ell$ for each $x$. First, for $x\geq 13$, we have
	\begin{equation*}
		f_1(x) = x^4 - 1 - 12x\left(x^2+1\right) > 0.
	\end{equation*}
	We also verify directly that $f_2(5)=3924>0$, $f_2(7)=73548>0$, and $f_{2}(11)=1505844>0$, so we see that the base case for the induction holds for each $p$.
	Next
	\begin{equation*}
		f_{\ell+1}(x) = x^2f_\ell(x) + 6(\ell+1)x^{\ell+1}\left(x^2+1\right)(x-1)+x^2-1-6x^{\ell+1}\left(x^2+1\right).
	\end{equation*}
	The first term is positive by induction. Using $x\ge2$, we show that the remaining terms are non-negative. This proves \eqref{E:claimcases} for $n\equiv1\pmod{12}$.

	Next assume that $n\equiv5\pmod{12}$. Simplifying \eqref{cusp} in this case, the $n$-th Fourier coefficient of the cuspidal part is $-\frac{3}{7}a_{1}(n)-\frac{3}{8}a_{2}(n)$. Then \Cref{lem:Deligne} implies that the absolute value is bounded by  $\frac{45}{56}d(n)n$. With the same argument as above, we may conclude the claim if
	{\setlength{\abovedisplayskip}{8pt}\setlength{\belowdisplayskip}{8pt}%
		\begin{equation}\label{ljinequality}
			\prod_{j=1}^{r}(\ell_{j}+1) \le \frac{56}{315} \prod_{j=1}^{r} \frac{p_{j}^{2\ell_{j}+2}-1}{p_{j}^{\ell_{j}}\left(p_{j}^{2}+1\right)}. 
	\end{equation}}%
	If \eqref{ljinequality} fails, then, since $\frac{56}{315}>\frac{1}{6}$, \eqref{eqn:ratbound} does not hold. As shown for $n\equiv 1\pmod{12}$, if \eqref{eqn:ratbound} fails, then $n\in\{5,7,11\}$. Thus for $n>5$ with $n\equiv 5\pmod{12}$, we have  $s_3(\frac{n-1}{4})=\sgn(A(n))$, as claimed. For $n=5$, a direct computer calculation shows that $s_3(1)=-1=\sgn(A(5))$. Note that this is the only possible exceptional case satisfying $n\equiv5\pmod{12}$.

	We finally consider $n\equiv9\pmod{12}$. Then, by \eqref{cusp}, the Fourier coefficient of the cusp form is $-\frac{12}{7}a_{1}(n)-\frac{3}{16}a_{2}(n)$. By \Cref{lem:Deligne}, the absolute value of this can be bounded by $\frac{213}{112}d(n)n$. Writing $n=3^{\ell}m$ and comparing \eqref{eqn:extrafactor} with the Fourier coefficients of the Eisenstein series in the other congruence classes and simplifying, the Fourier coefficients of the Eisenstein series have the extra factor $\frac{1}{5}(3^{2\ell+1}+7(-1)^{\ell+1}).$ Bounding as above, $s_3(\frac{n-1}{4})=\sgn(A(n))$ if
	{\setlength{\abovedisplayskip}{8pt}\setlength{\belowdisplayskip}{8pt}
	\begin{equation*}
		\frac{3^{2\ell+1}+7(-1)^{\ell}}{35} \prod_{p\mid m} \frac{p^{2\ord_p(n)+2}-1}{p^2+1} \ge \frac{213}{112} d(n)n.
	\end{equation*}}

\noindent Since $n=3^{\ell}m$ and $d(n)=(\ell+1)d(m)$, this is equivalent to
{\setlength{\abovedisplayskip}{8pt}\setlength{\belowdisplayskip}{8pt}
	\begin{equation*}
	 \prod_{p\mid m} \frac{p^{2\ord_p(n)+2}-1}{p^2+1} \ge \frac{213}{112} \frac{35(\ell+1)3^{\ell}}{3^{2\ell+1}+7(-1)^{\ell}}d(m)m.
	\end{equation*}}

Note that%
{\setlength{\abovedisplayskip}{8pt}\setlength{\belowdisplayskip}{6pt}%
	\setlength{\abovedisplayshortskip}{8pt}\setlength{\belowdisplayshortskip}{6pt}%
	\begin{equation*}
		\frac{(\ell+1)3^{\ell}}{3^{2\ell+1}+7(-1)^{\ell}}\leq
		\begin{cases}
			\frac{3}{10}&\text{if }\ell=1,\\
			\frac{27}{250}&\text{if }\ell=2,\\
			\frac{27}{545}&\text{if }\ell=3,\\
			\frac{81}{3938}&\text{if }\ell=4,\\
			\frac{729}{88750}&\text{if }\ell\geq 5.
		\end{cases}
\end{equation*}}
We then extend \eqref{E:claimcases} with
\begin{equation}\label{eqn:claimcases2}
\frac{p^{2\ell+2}-1}{p^{\ell} \left(p^2+1\right)} \ge
\begin{cases}
			6.4(\ell + 1) & \text{for $p=13$ and $\ell=1$},\\
			8(\ell + 1) & \text{for $p=5$ and $\ell=2$},\\
			8.4(\ell + 1) & \text{for $p=17$ and $\ell=1$},\\
			9.4(\ell + 1) & \text{for $p=19$ and $\ell=1$},\\
			11.4(\ell + 1) & \text{for $p=23$ and $\ell=1$},\\
			14.4(\ell + 1) & \text{for $p=29$ and $\ell=1$},\\
			15.4(\ell + 1) & \text{for $p=31$ and $\ell=1$},\\
			16(\ell + 1) & \text{for $p=7$ and $\ell=2$},\\
			18.4(\ell + 1) & \text{for $p=37$ and $\ell=1$},\\
			20(\ell+1) & \text{for $p\ge41$ or ($11\leq p\leq 37$ and $\ell\ge2$)}\\&\hspace{1.7cm} \text{ or ($p\in\{5,7\}$ and $\ell\geq 3$)}.
		\end{cases}
\end{equation}
Combining \eqref{eqn:claimcases2} with \eqref{E:claimcases}, we then conclude that for $n\equiv 9\pmod{12}$ with \[n\notin\{9, 21,33,45, 57,69,81,93,105,117,165\},\] we have $s_3(\frac{n-1}{4})=\sgn(A(n))$. A computer calculation for $n\leq 165$ establishes the remaining cases.
\end{proof}

\section{Proof of Theorem \ref{thm:M=8}}\label{sec:M=8}
\subsection{Proof of Theorem \ref{thm:M=8} (1)}
Let $b_{4}(n) := C_{1^{4}2^{2}4^{-2}}(n)$ and $s_{4}:=\sgn(b_{4}(n))$. We first obtain a formula for the generating function of $b_{4}(n)$.
\begin{lem}\label{lem:14224-2}
	We have
	\[
		\frac{\eta(z)^4\eta(2z)^2}{\eta(4z)^2}=1-4\sum_{n\geq 1} \left(\frac{-4}{n}\right)\sigma(n)q^n +8\sum_{n\geq 1} (-1)^n\sigma(n)q^{4n} -32\sum_{n\geq 1} \sigma(n)q^{16n}.
	\]
\end{lem}
\begin{proof}
By \cite[Theorem 1.64]{OnoBook}, $\frac{\eta(z)^4\eta(2z)^2}{\eta(4z)^2}\in M_2(\Gamma_0(16))$. The right-hand side of Lemma \ref{lem:14224-2} is
\begin{equation}\label{eqn:E2rewrite}
-2E_{2,\chi_{-4},\chi_{-4}} + \frac{1}{3}E_2\big|V_4 - \frac{2}{3}E_2\big| U_{2}\circ V_8 +\frac{4}{3}E_2\big|V_{16}.
\end{equation}
One checks that \eqref{eqn:E2rewrite} coincides with the corresponding identity with $E_2$ replaced by $\widehat E_2$. Thus, \eqref{eqn:E2rewrite} is modular. By Lemmas \ref{lem:Ekchipsi} and  \ref{lem:operators} (1), (4), \eqref{eqn:E2rewrite} lies in $M_2(\Gamma_0(16))$.  Thus both sides are elements of $M_2(\Gamma_0(16))$. By Lemma \ref{lem:valence}, it suffices to verify it for the first $4$ Fourier coefficients, which has been verified computationally.
\end{proof}
\begin{proof}[Proof of Theorem \ref{thm:M=8} {\rm(1)}]
For $n=0$, Lemma \ref{lem:14224-2} gives $b_{4}(0)=1>0$. We next let $n\in\N$ and split into cases depending on $\ord_2(n)$. By Lemma \ref{lem:14224-2}, if $n$ is odd, then
	\[
		b_{4}(n)=-4\left(\frac{-1}{n}\right) \sigma(n).
	\]
	Since $4\sigma(n)>0$, this gives
	\[
		s_{4}(n)=-\left(\frac{-1}{n}\right)=
		\begin{cases}
			1&\text{if }n\equiv 3\pmod{4},\\
			-1&\text{if }n\equiv 1\pmod{4}.
		\end{cases}
	\]
	This gives the claim for $n$ odd.

	For $n\equiv 2\pmod{4}$, we see directly that none of the sums in \Cref{lem:14224-2} contributes to the $n$-th Fourier coefficient. This gives the claim in this case.

	Next suppose that $4\mid n$ but $16\nmid n$. In this case, only the second sum in \Cref{lem:14224-2} contributes to the $n$-th Fourier coefficient, so we obtain
	\[
		b_{4}(n)=-8 \sigma\left(\frac{n}{4}\right).
	\]
	This implies that $s_{4}(n)=-1$. This gives the corresponding case.

	Finally, if $16\mid n$, then the final two sums in \Cref{lem:14224-2} contribute to the $n$-th Fourier-coefficient, and we have
	\begin{equation*}
		b_{4}(n) = 8\left(\sigma\left(\frac{n}{4}\right)-4\sigma\left(\frac{n}{16}\right)\right).
	\end{equation*}
	We next write $n=2^{\ell}m$ with $\ell\in\N_{\geq3}$ and $m$ odd and use multiplicativity of $\sigma(n)$ to simplify
	\begin{equation*}
		\sigma\left(\frac{n}{4}\right)-4\sigma\left(\frac{n}{16}\right) = 3\sigma(m).
	\end{equation*}
	Thus, for $16\mid n$, we obtain
	\[
		b_{4}(n)=24\sigma(m)>0.
	\]
Hence $s_{4}(n)=1$ in this case, completing the proof.
\end{proof}
\subsection{Proof of Theorem \ref{thm:M=8} (2)}

For $\bm{a}\in\N^{\ell}$, set
\[
r_{\bm{a}}(n):=\#\Big\{\bm{n}\in\Z^{\ell}: \sum_{j=1}^{\ell} a_j n_j^2=n\Big\}.
\]
Using \cite[Theorem 1.62, Propositions 1.41 and 3.7 (1)]{OnoBook}, Lemmas \ref{lem:operatorshalf} and \ref{lem:valence}, we obtain the following identity.
\begin{lem}\label{lem:142443}
	We have
	\begin{multline*}
		\frac{\eta(z)^4\eta(2z)^4}{\eta(4z)^3} = -\sum_{n\equiv 1\pmod{4}} r_{(1,1,2,2,2)}(n) q^n+ \sum_{n\equiv 3\pmod{4}} r_{(1,1,2,2,2)}(n)q^n\\
		-\frac{1}{5}\sum_{n\equiv 2\pmod{16}} r_{(1,1,2,2,2)}(n) q^n-\frac{1}{5}\sum_{n\equiv 4\pmod{16}} r_{(1,1,2,2,2)}(n) q^n
		+\frac{1}{5}\sum_{n\equiv 6\pmod{16}} r_{(1,1,2,2,2)}(n) q^n\\-\frac{3}{7}\sum_{n\equiv 10\pmod{16}} r_{(1,1,2,2,2)}(n) q^n
		-\frac{1}{5}\sum_{n\equiv 12\pmod{16}} r_{(1,1,2,2,2)}(n) q^n+\frac{1}{5}\sum_{n\equiv 14\pmod{16}} r_{(1,1,2,2,2)}(n) q^n\\
		+\frac{1}{5}\sum_{n\equiv 0\pmod{8}} \left(r_{(1,1,2,2,2)}(n)+4r_{(1,1,2,2,2)}\left(\frac{n}{4}\right)\right) q^n.
	\end{multline*}
\end{lem}

\begin{proof}[Proof of Theorem \ref{thm:M=8} {\rm(2)}]
Liouville \cite{Liouville2} proved that $r_{(1,1,2,2)}(n)>0$ for every $n\in\N$. Since $r_{(1,1,2,2,2)}(n)\!\geq\! r_{(1,1,2,2)}(n)$, obtained by setting the fifth variable to zero,  we have $r_{(1,1,2,2,2)}(n)\!>0$ for all $n \in \mathbb{N}$. Theorem \ref{thm:M=8} (2) now follows immediately from Lemma \ref{lem:142443}.
\end{proof}

\section{Proof of \Cref{thm:QpPpgen}, Corollary \ref{cor:conj99}, and \Cref{thm:ClassNum}} \label{sec:conj99}
\subsection{Proof of Theorem \ref{thm:QpPpgen}}
We next bound and determine the sign of $\sum_{d\mid n}(\frac{d}{p}) d^k$.
\begin{lem}\label{lem:sumdbound}
For $n=p^am$ with $p\nmid m$ and $k\in\N$, we have
\[
\left|\sum_{d\mid n}\left(\frac{d}{p}\right) d^k\right|\gg_{\varepsilon} m^{k-\varepsilon}, \quad \sgn\left(\sum_{d\mid n}\left(\frac{d}{p}\right) d^k\right)=\left(\frac{m}{p}\right).
\]
\end{lem}
\begin{proof}
	By multiplicativity, we have
	\begin{equation*}
		\sum_{d\mid n} \left(\frac dp\right) d^k = \prod_{\ell \mid n} \sum_{j=0}^{\ord_{\ell}(n)} \left(\frac{\ell}{p}\right)^j \ell^{k j}.
	\end{equation*}
	Since $k\geq 1$, for $\ell\neq p$ and $r\in\N$ we have 
	\[
		\left|\sum_{j=0}^{r} \left(\frac{\ell}{p}\right)^j \ell^{k j}\right|\geq  \ell^{kr}-\ell^{k(r-1)}.
	\]
	Thus, for any $\varepsilon>0$, we have, using \cite[(3.27)]{RosserSchoenfeld} in the final step,
	\[
		\left|\sum_{d\mid n} \left(\frac dp\right) d^k \right|\geq m^k \prod_{\ell\mid  m}\left(1- \ell^{-k}\right)\gg_{\varepsilon} m^{k-\varepsilon}.
	\]
	This proves the first claim. For the second claim, we compute
	\begin{equation*}
		\sum_{j=0}^r \left(\frac{\ell}{p}\right)^j \ell^{kj} =
		\begin{cases}
			\frac{\ell^{(r+1)k}-1}{\ell^k-1} & \text{if $\left(\frac{\ell}{p}\right)=1$},\\
			\noalign{\vskip4pt}
			\frac{\ell^{(r+1)k}+1}{\ell^k+1} & \text{if $\left(\frac{\ell}{p}\right)=-1$, $2\mid r$},\\
			\noalign{\vskip4pt}
			-\frac{\ell^{(r+1)k}-1}{\ell^k+1} & \text{if $\left(\frac{\ell}{p}\right)=-1$, $2\nmid r$},\\
			\noalign{\vskip4pt}
			1 & \text{if $\ell=p$}.
		\end{cases}
	\end{equation*}
	Since $k\geq 1$, this is positive unless $(\frac{\ell}{p})=-1$ and $r$ is odd, in which case it is negative. Thus
	\begin{equation*}
		\sgn\left(\sum_{d\mid n} \left(\frac dp\right) d^k\right) = (-1)^{\#\left\{\text{$\ell$ prime}\,:\,\left(\frac{\ell}{p}\right)=-1,\,\text{$2\nmid\ord_{\ell}(m)$}\right\}}.
	\end{equation*}
	Now write $n=p^am$ with $a\in\N_0$ and $\gcd(p,m)=1$. The claim follows since
	\[
		\left(\frac{m}{p}\right)=\prod_{\ell^r\| m}\left(\frac{\ell^{r}}{p}\right) = (-1)^{
		\#\left\{\text{$\ell$ prime}\,:\,\left(\frac{\ell}{p}\right)=-1,\,\text{$2\nmid\ord_{\ell}(m)$}\right\}}.\qedhere
	\]
\end{proof}
\begin{proof}[Proof of Theorem \ref{thm:QpPpgen}]
	By \cite[Theorem 1.64]{OnoBook}, we have $Q_p\in M_{\frac{p^2-p}{2}}(\Gamma_0(p),\chi_{(\frac{-1}{p})p})$ and $P_p\in M_{\frac{p-1}{2}}(\Gamma_0(p),\chi_{(\frac{-1}{p})p})$, so \Cref{lem:Eisensteinp-growth} yields that there exist $f_p\in S_{pk_p}(\Gamma_0(p),\chi_{(\frac{-1}{p})p})$ and $g_p\in S_{k_p} $ $(\Gamma_0(p),\chi_{(\frac{-1}{p})p})$ such that
\begin{equation}\label{eqn:QpPpEisenstein}
Q_p=L_{pk_p,p} E_{pk_p,1,\chi_{\left(\frac{-1}{p}\right)p}}+f_p,\qquad P_p=L_{k_p,p} E_{k_p,1,\chi_{\left(\frac{-1}{p}\right)p}}+g_p.
\end{equation}
 We split the Fourier expansions naturally into the two pieces 
\begin{align}\label{eqn:QpFouriersplit}
Q_p(z)&=:\sum_{n\geq 0} c_{Q_{p}}(n) q^n=:\sum_{n\geq 0}A_{\text{E}}(n) q^n +\sum_{n\geq 1}c_{f_{p}}(n) q^n,\\
\label{eqn:PpFouriersplit}
P_p(z)&=:\sum_{n\geq 0} c_{P_{p}}(n) q^n=:\sum_{n\geq 0}B_{\text{E}}(n)  q^n +\sum_{n\geq 1}c_{g_{p}}(n) q^n
\end{align}
corresponding to the right-hand sides of \eqref{eqn:QpPpEisenstein}.

\noindent(1) Employing \Cref{lem:SY} and noting that $f_p$ is uniquely determined by $p$, we have, recalling that $n=p^am$ with $p\nmid m$,
 \begin{equation*}
 	|c_{f_{p}}(n)|\ll_{p,\varepsilon} \left\|f_p\right\| n^{\frac{pk_p-1}{2}+\varepsilon}\ll_p n^{\frac{pk_p-1}{2}+\varepsilon}\ll_{a,p,\varepsilon} m^{\frac{pk_p-1}{2}+\varepsilon}.
 \end{equation*}
 Using the bound in \Cref{lem:sumdbound}, we conclude from \eqref{eqn:QpFouriersplit} that,  for $m$ sufficiently large, 
\begin{equation}\label{eqn:sgnaQp}
\sgn\left(c_{Q_{p}}(n)\right)=\sgn\left(A_{\text{E}}(n)\right).
\end{equation}
 By the evaluation of the sign in \Cref{lem:sumdbound}, we then obtain
\[
\sgn\left(A_{\text{E}}(n)\right)=\sgn\left(L_{pk_p,p}\right)\left(\frac{m}{p}\right).
\]
Finally, using Lemma \ref{lem:Lchisgn}, we conclude that
\begin{equation}\label{eqn:sgnaQpEisen}
\sgn\left(A_\text{E}(n)\right)=\left(\frac{2}{p}\right)\left(\frac{m}{p}\right).
\end{equation}
This finishes the proof of part (1).

\noindent(2) By \Cref{lem:SY}, we have
$|c_{g_p}(n)|\ll_{a,p,\varepsilon} m^{\frac{k_p-1}{2}+\varepsilon}$. As in part (1), \Cref{lem:sumdbound} and \eqref{eqn:PpFouriersplit} hence imply that, for $m$ sufficiently large, we have
\begin{equation}\label{eqn:sgnaPp}
\sgn\left(c_{P_p}(n)\right)=\sgn\left(B_\text{E}(n)\right)=\sgn\left(L_{k_p,p}\right) \left(\frac{m}{p}\right).
\end{equation}
Finally, Lemma \ref{lem:Lchisgn} implies that
\begin{equation}\label{eqn:sgnaPpEisen}
\sgn\left(B_\text{E}(n)\right)=\left(\frac{-2}{p}\right)\left(\frac{m}{p}\right),
\end{equation}
finishing the proof.
\end{proof}

\subsection{Proof of Corollary \ref{cor:conj99}}
Since (see \cite{LMFDB})
\begin{equation}\label{eqn:S3vanish}
S_{3}(\Gamma_0(3),\chi_{-3})=\{0\},
\end{equation}
\eqref{eqn:QpPpEisenstein} and \Cref{lem:Dirichlet} yield the following identity for $b_5(n):=C_{1^93^{-3}}(n)$.

\begin{lem}\label{lem:1933}
	We have
	 \[
	 	\frac{\eta(z)^9}{\eta(3z)^3} = 1 - 9\sum_{n\ge1} \sum_{d\mid n} \left(\frac d3\right) d^2 q^n.
	 \]
	In particular, for $n\in\N$ we have
	\[
		b_{5}(n)=-9\sum_{d\mid n} \left(\frac d3\right) d^2.
	\]
\end{lem}

We next prove Corollary \ref{cor:conj99} (1).  Let $s_{5}(n):=\sgn(b_{5}(n))$ and $s_6(n):=\sgn(C_{1^55^{-1}}(n))$.
\begin{proof}[Proof of Corollary \ref{cor:conj99}]
	(1) Since  $f_3=0$ in \eqref{eqn:QpPpEisenstein}
by \eqref{eqn:S3vanish}, \eqref{eqn:QpFouriersplit} implies that \eqref{eqn:sgnaQp} holds for $n\in\N$, and hence \eqref{eqn:sgnaQpEisen} gives
	\[
		s_5(n)=\sgn\left(c_{Q_3}(n)\right)=\sgn\left(A_\text{E}(n)\right)=\left(\frac{2}{3}\right)\left(\frac{m}{3}\right)=-\left(\frac{m}{3}\right).
	\]

	\noindent (2) Since $S_{2}(\Gamma_0(5),\chi_5)=\{0\}$ (see \cite{LMFDB}), we similarly conclude from \eqref{eqn:PpFouriersplit} that \eqref{eqn:sgnaPp} holds for $n\in\N$ and thus \eqref{eqn:sgnaPpEisen} yields that $s_6(n)=-(\frac{m}{5})$.
\end{proof}

\subsection{Proof of Theorem \ref{thm:ClassNum}}
The first step in proving Theorem \ref{thm:ClassNum} is to write our functions in terms of class number generating functions. We have the following.
\begin{lem}\label{lem:etaHrel}
 	We have
 	\begin{equation*}
 		\frac{\eta(8z)^2\eta(16z)^2}{\eta(24z)} = (\mathcal H_{4,3} - \mathcal H_{1,3}) \mid S_{24,1}(z) - \frac12 (\mathcal H_{4,3} - \mathcal H_{1,3}) \mid S_{24,17}(z) - (\mathcal H_{4,3} +2\mathcal H_{1,3}) \mid S_{24,9}(z).
 	\end{equation*}
\end{lem}
\begin{proof}
	By \Cref{lem:Hslash}, $\mathcal H_{1,3}\in M_\frac32(\Gamma_0(12),\chi_{12})$ and $\mathcal H_{4,3}\in M_\frac32(\Gamma_0(24),\chi_{12})$. By \Cref{lem:operatorshalf} (2),  the right-hand side of the lemma lies in $M_\frac32(\Gamma_0(576),\chi_{12})$. Moreover, since $\eta(24z)\in M_\frac12(\Gamma_0(576),\chi_{12})$ (see \cite[Corollary 1.62]{OnoBook}), $\eta(24z)$ times the right-hand side is in $M_2(\Gamma_0(576))$. Next, by \cite[Theorem 1.64]{OnoBook}, $\eta(8z)^2\eta(16z)^2\in M_2(\Gamma_0(64))$. Thus $\eta(24z)$ times the difference of the left- and right-hand sides lies in $M_2(\Gamma_0(576))$. By \Cref{lem:valence}, we have to check $192$ Fourier coefficients. The claim was verified computationally by checking the first $192$ Fourier coefficients.
\end{proof}
\begin{proof}[Proof of Theorem \ref{thm:ClassNum}]
	Let $b_{7}(n):=C_{1^{2}2^{2}3^{-1}}(n)$ and $s_{7}(n):=\sgn(b_{7}(n))$. Let
 	\begin{equation*}
 		\frac{\eta(8z)^2\eta(16z)^2}{\eta(24z)} =: \sum_{n\ge0} C(n) q^n.
 	\end{equation*}
	Then $b_{7}(n)=C(8n+1)$. Thus
	\begin{equation*}
	 	s_{7}(n) = \sgn(C(8n+1)).
	\end{equation*}
	By \eqref{eqn:Hslashdef}, we have
	\begin{equation*}
	 	\mathcal H_{\ell_1,\ell_2}(z) = \sum_{n\ge0} \left(H(\ell_1\ell_2n) - \ell_2H\left(\frac{\ell_1n}{\ell_2}\right)\right) q^n.
	\end{equation*}
	Hence \Cref{lem:etaHrel} implies that, for $n=3\ell$, we have 
	\begin{equation*}
	 	b_{7}(3\ell) = H(12(24\ell+1))-H(3(24\ell+1)).
	\end{equation*}
	We write $3(24\ell+1) = Df^2$ and $12(24\ell+1) = D(2f)^2$ with $-D$ a fundamental discriminant and $f\in\N$. Then \eqref{E:Hform} implies that
	\begin{equation*}
		H(3(24\ell+1))=H(D)S_D(f), \quad H(12(24\ell+1)) = H(D) S_D(2f).
	\end{equation*}
	Since $S_D(f)$ is multiplicative and $f$ is odd, we have
	\[
		S_{D}(2f)-S_D(f)= S_D(f)S_D(2)-S_D(f)=S_D(f)\left(\sigma(2)-\left(\frac{-D}2\right)-1\right)>0.
	\]
	So $s_7(3\ell)=1=(\frac{8\cdot n +1}{3})$, and we are done in this case. Similarly $b_{7}(3\ell+2)=C(24\ell+17)$. The same argument gives $s_7(3\ell+2)=-1=(\frac{8n+1}{3})$ for $n=3\ell+2$. This finishes the case $a=0$, as $3\mid(8n+1)$ iff $n\equiv 1\pmod{3}$, where we write $8n+1=3^am$ with $3\nmid m$ as in the theorem.

We hence next assume that $n=3\ell+1$. In this case, \Cref{lem:etaHrel} implies that
\[
		b_{7}(3\ell+1) = C(24\ell+9)= -H(36(8\ell+3))+ 3H\left(4(8\ell+3)\right) - 2H(9(8\ell+3)) + 6H(8\ell+3).
\]
Write $3^{a-1}m=8\ell+3=Df^2$ with $-D$ fundamental and $f\in\N$. Using \eqref{E:Hform} gives 
	\begin{equation}\label{eqn:C24n+9}
	 	C(24\ell+9) = -H(D) (S_D(6f) - 3S_D(2f) + 2S_D(3f) - 6S_D(f)).
	\end{equation}
	Note that $2\nmid f$. Since $D$ is a fundamental discriminant, we have $9\nmid D$, so $a\geq 3$ if and only if $3\mid f$. Moreover, $9\mid (8\ell+3)$ if and only if $3\mid f$. Therefore
	\begin{equation*}
	 	a\geq 3\Leftrightarrow 3\mid f \Leftrightarrow 8\ell+3 \equiv 0 \pmod9
	 	\Leftrightarrow \ell \equiv 3 \pmod9.
	\end{equation*}
	
We next assume that $a\in\{1,2\}$, and hence $\ell\not\equiv3\pmod9$. Since $S_D(f)$ is multiplicative, using \eqref{eqn:SDfdef} and \eqref{E:Hform}, we have, by \eqref{eqn:C24n+9},
	\begin{align*}
	 	C(24\ell+9) &= -H(D) S_D(f) (S_D(6) - 3S_D(2) + 2S_D(3) - 6)\\
	 	&= -H(8\ell+3)\bigg(\sigma(6) - \left(\frac{-D}3\right) \sigma(2) - \left(\frac{-D}2\right) \sigma(3) + \left(\frac{-D}6\right) - 3\sigma(2)\\
	 	&\hspace{6cm} + 3\left(\frac{-D}2\right) + 2\sigma(3) - 2\left(\frac{-D}3\right) - 6\bigg).
	\end{align*}
	Note that $(\frac{-D}2)=(\frac{-3}2)=-1$ since $D\equiv3\pmod8$ and  $(\frac{-D}3)=(\frac{-8\ell}3)=(\frac n3)$. Thus
	\begin{equation*}
	 	C(24\ell+9) = -6H(8\ell+3) \left(1-\left(\frac{\ell}{3}\right)\right).
	\end{equation*}
	This proves the claim for $a\in\{1,2\}$.

Finally, suppose that $a\geq 3$, which is equivalent to $n\equiv3\pmod9$. We write $f=3^rg$ with $3\nmid g$. By \eqref{eqn:C24n+9} and the multiplicativity of $S_D(f)$ (splitting off the $3$-powers), we have
	\begin{equation*}
	 	C(24\ell+9) = -H(D) S_D(g) \left(S_D\left(2\cdot3^{r+1}\right) - 3S_D\left(2\cdot3^r\right) + 2S_D\left(3^{r+1}\right) - 6S_D\left(3^r\right)\right).
	\end{equation*}
	We compute
	\begin{equation*}
	 	S_D\left(3^{\nu}\right) = \sigma\left(3^{\nu}\right) - \left(\frac{-D}3\right) \sigma\left(3^{\nu-1}\right), \quad S_D(2) = \sigma(2) - \left(\frac{-D}2\right) = 3 - (-1) = 4.
	\end{equation*}
	Simplifying gives
	\begin{equation*}
		-H(D) S_D(g) \left(6S_D\left(3^{r+1}\right) - 18S_D\left(3^r\right)\right) = -6H(D) S_D(g) \left(1-\left(\frac{-D}3\right)\right).
	\end{equation*}
Noting that the right-hand side is independent of $r$, the result follows.
\end{proof}




\begin{thebibliography}{99}
\bibitem{AndrewsBorwein} G. Andrews, \begin{it}On a conjecture of Peter Borwein\end{it}, J. Symb. Comput. \textbf{20} (1995), 487--501.
\bibitem{AndrewsLewis}G. Andrews and R. Lewis, \begin{it}The ranks and cranks of partitions moduli $2$, $3$, and $4$,\end{it} J. Number Theory \textbf{85} (2000), 74--84.
\bibitem{Apostol}T. Apostol, \begin{it}Introduction to analytic number theory\end{it}, Springer-Verlag, 1976.
\bibitem{Borwein}P. Borwein, \begin{it}Some restricted partition functions\end{it}, J. Number Theory \textbf{45} (1993), 228--240.
\bibitem{BosmaKane}W. Bosma and B. Kane, \begin{it}The triangular theorem of eight and representation by quadratic polynomials\end{it}, Proc. Amer. Math. Soc. \textbf{141} (2013), 1473--1486.
\bibitem{BGHK-WeakSigns}K. Bringmann, G. Han, B. Heim, and B. Kane, \begin{it}Periodic sign changes for weakly holomorphic eta-quotients\end{it}, Int. J. Number Theory \textbf{21} (2025), 2111--2139.
\bibitem{BK} K. Bringmann and B. Kane, {\it Class numbers and representations by ternary quadratic forms with congruence conditions}, Mathematics of Computation \textbf{91} (2022), 295--329.
\bibitem{BruinierKohnen} J. Bruinier and W. Kohnen, \begin{it}Sign changes of coefficients of half integral weight modular forms\end{it}, In: Modular forms on Schiermonnikoog (eds. B. Edixhoven et. al.), 57–66, Cambridge Univ. Press, 2008.
\bibitem{CG} R. Chen and G. Garvan,  \begin{it}Congruences modulo $4$ for weight $3/2$ eta-products\end{it}, Bull. Aust. Math. Soc. \textbf{103} (3) (2021), 405–417.
\bibitem{Cohen}H. Cohen, \begin{it}Sums involving the values at negative integers of $L$-functions of quadratic characters\end{it}, Math. Ann. \textbf{217} (1975), 271--285.
\bibitem{CGR06} S. Cooper, S. Gun, and B. Ramakrishnan, \begin{it} On the lacunarity of two-eta-products\end{it}, Georgian Math. Journal \textbf{13} (4), (2006), 659--673.
\bibitem{Deligne} P. Deligne, \begin{it}La conjecture de Weil I,\end{it} Inst. Hautes \'Etudes Sci. Publ. Math. \textbf{43} (1974), 273--307.
\bibitem{DiamondShurman} F. Diamond and J. Shurman, \begin{it}A first course in modular forms\end{it}, Graduate Texts in Math. \textbf{228}, Springer, 2005.
\bibitem{HZ} F. Hirzebruch and D. Zagier, \begin{it}Intersection numbers of curves on Hilbert modular surfaces and modular forms of Nebentypus\end{it}, Invent. Math. \textbf{36} (1976), 57--113.
\bibitem{KaneResolution}D. Kane, \begin{it}Resolution of a conjecture of Andrews and Lewis involving cranks of partitions\end{it}, Proc. Amer. Math. Soc. \textbf{132} (2004), 2247--2256.
\bibitem{KnoppKohnenPribitkin}M. Knopp, W. Kohnen, and W. Pribitkin, \begin{it}On the sign of Fourier coefficients of cusp forms\end{it}, Ramanujan J. \textbf{7} (2003), 269--277.
\bibitem{Koehlerbook} G. K\"ohler, \begin{it}Eta products and theta series identities \end{it} Springer Monographs in Mathematics, Springer, 2011.


\bibitem{KLSW10} E. Kowalski, Y.-K Lau, K. Soundararajan, and J. Wu, \begin{it}On modular signs\end{it}, Math. Proc. Camb. Phil. Soc. \textbf{149} (2010), 389--411.

\bibitem{LMFDB} The {LMFDB collaboration}, \begin{it}The $L$-functions and modular forms database\end{it}, \url{https://www.lmfdb.org}, online, accessed 26 July, 2024.
\bibitem{KLW} W. Kohnen, Y. Lau, and J. Wu, \begin{it}Fourier coefficients of cusp forms of half-integral weight\end{it}, Math. Z. \textbf{273} (2013), 29--41.

\bibitem{Liouville2}J. Liouville, \begin{it}Sur la forme $x^2+y^2+2(z^2+t^2)$\end{it}, J. Math. Pures Appl. \textbf{5} (1860), 269--272.

\bibitem{Liouville}J. Liouville, \begin{it}Nouveaux th\'eor\`emes concernant les nombres triangulaires\end{it}, J. Math. Pures Appl. \textbf{8} (1863), 73--84.



\bibitem{OnoBook}K. Ono, \begin{it}The web of modularity: arithmetic of the coefficients of modular forms and $q$-series\end{it}, Regional Conference Series in mathematics, Amer. Math. Soc. \textbf{102}, 2004.

	\bibitem{RosserSchoenfeld} J. Rosser and L. Schoenfeld, \begin{it}Approximate formulas for some functions of prime numbers\end{it} Ill. J. Math. \textbf{6} (1962), 64--94.

	\bibitem{Z} M. Schlosser and N. Zhou, {\it On the infinite Borwein product raised to a positive real power}, Ramanujan J. \textbf{61} (2023), 515--543.
\bibitem{SY}R. Schulze-Pillot and A. Yenirce, \begin{it}Petersson products of bases of spaces of cusp forms and estimates for Fourier coefficients\end{it}, Int. J. Number Theory \textbf{14} (2018), 2277--2290.

\bibitem{Zagier}D. Zagier, \begin{it}Nombres de classes et formes modulaires de poids $3/2$\end{it}, C.R. Acad. Sci. Paris (A) \textbf{281} (1975), 883--886.
\bibitem{ZagierMellin}D. Zagier, \begin{it}The Mellin transform and other useful analytic techniques\end{it}, Appendix to E. Zeidler, Quantum Field Theory I: Basics in mathematics and physics. A bridge between mathematicians and physicists, Springer-Verlag, Berlin-Heidelberg-New York (2006), 305--323.
\end{thebibliography}
\end{document}